\documentclass[12pt]{amsart}
\usepackage{ amssymb,amsmath,amsfonts,amsthm,nomencl,mathrsfs} 
\usepackage[arrow, matrix, curve]{xy}
\usepackage[latin1]{inputenc}
\usepackage{a4wide}
\def\Ric{{\operatorname{Ric}}}
\newcommand{\E}{\mathbb{E}}

\newcommand{\IC}{\mathbb{C}}
\newcommand{\IR}{\mathbb{R}}

\newcommand{\question}[1]{\leavevmode{\marginpar{\tiny%
$\hbox to 0mm{\hspace*{-0.5mm}$\leftarrow$\hss}%
\vcenter{\vrule depth 0.1mm height 0.1mm width \the\marginparwidth}%
\hbox to 0mm{\hss$\rightarrow$\hspace*{-0.5mm}}$\\\relax\raggedright #1}}}

\newcommand{\IT}{T}

\newcommand{\IMM}{\mathscr{M}}

\newcommand{\ICC}{C^{\infty}}

\newcommand{\ILL}{\mathscr{L}}

\newcommand{\IJJ}{\mathscr{J}}

\newcommand{\IHH}{\mathscr{H}}

\newcommand{\IAA}{\mathscr{A}}

\renewcommand{\c}{ \mathrm{c} }

\newcommand{\IL}{L}

\newcommand{\dom}{\mathrm{Dom}}

\newcommand{\Id}{{d}}

\newcommand{\f}{\frac}

\newcommand{\slim}{\mathop{\rm{st}\rule[0.5ex]{1ex}{0.1ex}\mathrm{lim}}}
\def\partr#1{/\!/_{\!#1}^{\phantom{.}}}

\newcommand\newdot{{\kern.8pt\cdot\kern.8pt}}
\def\nbull{{\raise1.5pt\hbox{\bf .}}}

\theoremstyle{plain}            
\newtheorem{theorem}{theorem}[section]
\newtheorem{Lemma}[theorem]{Lemma}

\newtheorem{Theorem}[theorem]{Theorem}

\newtheorem{Theoreme}{Theorem}

\newtheorem{Corollarye}{Corollary}

\theoremstyle{definition}

\newtheorem{Remark}[theorem]{Remark}
\newtheorem{Example}[theorem]{Example}

\numberwithin{equation}{section}

\allowdisplaybreaks[1]

\begin{document}

\begin{titlepage}
\title[Scattering theory and Ricci flow]{Scattering theory without injectivity radius assumptions and spectral stability for the Ricci flow}

\author[B. G\"uneysu]{Batu G\"uneysu}

\author[A. Thalmaier]{Anton Thalmaier}

\address{Institut für Mathematik, Humboldt-Universität zu Berlin, 12489 Berlin, Germany}
\email{gueneysu@math.hu-berlin.de}

\address{Mathematics Research Unit, University of Luxembourg, Maison du Nombre, 4364 Esch-sur-Alzette, Luxembourg}
\email{anton.thalmaier@uni.lu}

\end{titlepage}

\maketitle 

\begin{abstract} We prove a completely new integral criterion for the existence and completeness of the wave operators $W_{\pm}(-\Delta_h,-\Delta_g, I_{g,h})$ corresponding to the (unique self-adjoint realizations of) the Laplace-Beltrami operators $-\Delta_j$, $j=1,2$, that are induced by two quasi-isometric complete Riemannian metrics $g$ and $h$ on an open manifold $M$. In particular, this result provides a criterion for the absolutely continuous spectra of $-\Delta_g$ and $-\Delta_h$ to coincide. Our proof relies on estimates that are obtained using a probabilistic Bismut type formula for the gradient of a heat semigroup.  Unlike all previous results, our integral criterion only requires some lower control on the Ricci curvatures and some upper control on the heat kernels, but no control at all on the injectivity radii. As a consequence, we obtain a stability result for the absolutely continuous spectrum under a Ricci flow. 
\end{abstract}

%

\section{Introduction}

As the (unique self-adjoint realization in $L^2(M,g)$ of the) Laplace-Beltrami operator $-\Delta_g\geq 0$ on a noncompact geodesically complete Riemannian manifold $(M,g)$ typically contains some continuous spectrum, a natural question that arises is to what extent one can control at least certain parts of the continuous spectrum. A particular decomposition of the spectrum $\mathrm{spec}(-\Delta_g)$ is given (cf. \cite{Re3, weidmann2} and the appendix of this paper) by
$$
\mathrm{spec}(-\Delta_g)=\mathrm{spec}_{\mathrm{ac}}(-\Delta_g)\bigcup \mathrm{spec}_{\mathrm{sc}}(-\Delta_g)\bigcup \mathrm{spec}_{\mathrm{pp}}(-\Delta_g),
$$
where 
\begin{itemize}
\item $\mathrm{spec}_{\mathrm{ac}}(-\Delta_g)$ denotes the absolutely continuous spectrum (cf.~Appendix~\ref{wave}).
\item $\mathrm{spec}_{\mathrm{sc}}(-\Delta_g)$ the singular continuous spectrum,
\item $\mathrm{spec}_{\mathrm{pp}}(-\Delta_g)$ the pure point spectrum,
\end{itemize}
so that
$$
\mathrm{spec}_{c}(-\Delta_g)= \mathrm{spec}_{\mathrm{ac}}(-\Delta_g)\bigcup \mathrm{spec}_{\mathrm{sc}}(-\Delta_g)
$$
is the whole continuous spectrum. The absolutely continuous spectrum corresponds to the quantum dynamics in the following sense: by definition, $\mathrm{spec}_{\mathrm{ac}}(-\Delta_g)$ is the spectrum of the restriction of $-\Delta_g$ to the closed subspace of $L^2(M,g)$ given by absolutely continuous states $\psi$ corresponding to $-\Delta_g$. But for those $\psi$\rq{}s the RAGE theorem \cite{Re3, weidmann2} shows 
\begin{align}\label{abso}
\lim_{|t|\to\infty} 1_{K}\exp(it\Delta_g)\psi =0\quad\text{ in $L^2(M,g)$ for every compact $K\subset M$.}
\end{align}
As by Schrödinger\rq{}s equation $\exp(it\Delta_g)\psi$ is the state at the time $t$ given that the initial state was $\psi$, property (\ref{abso}) shows that the quantum particle eventually leaves every compact set, if the initial state was an absolutely continuous one.\smallskip

 A perturbative approach to control $\mathrm{spec}(-\Delta_g)$ is provided by the machinery of $2$-Hilbert-space scattering theory: namely, assume that $h$ is another Riemannian metric on $M$ which is quasi-isometric to $g$ and whose absolutely continuous spectrum is known. Then, with the trivial identification map  
$$
I_{g,h}: \IL^2(M,g)\longrightarrow \IL^2(M,h),\>\>f\longmapsto f,
$$
one can ask whether the $2$-Hilbert-space wave operators 
$$
W_{\pm}(-\Delta_{h},-\Delta_{g}, I_{g,h})=\slim_{t\to\pm\infty}\exp(-it\Delta_{h})I_{g,h}\exp(it\Delta_{g})P_{\mathrm{ac}}(-\Delta_g)
$$
exist and are complete (cf. section \ref{wave} for the precise definitions), the point being that the latter property implies
$$
\mathrm{spec}_{\mathrm{ac}}(-\Delta_g)=\mathrm{spec}_{\mathrm{ac}}(-\Delta_h),
$$
and one has managed to transfer the known spectral information from $-\Delta_h$ to $-\Delta_g$. The current state of the art concerning criteria for the existence and completeness of the $W_{\pm}(-\Delta_{h},-\Delta_{g}, I_{g,h})$ is the main result from \cite{HPW}. There, Hempel, Post and Weder prove the following result:\smallskip

 \emph{Assume that for both $j\in \{g,h\}$ one has
\begin{align}\label{hpw}
\int  \gamma_j(x) \delta_{g,h}(x)d\mu_j(x)<\infty,
\end{align}
where 
\begin{itemize}
\item $\mu_j$ is the Riemannian volume measure,
\item $\gamma_j:M\to [0,\infty)$ is a certain explicitly function which depends in a monotonically decreasing way on a local lower bound of the injectivity radius and in a monotonically increasing way on a local lower bound on the Ricci curvature $\mathrm{Ric}_j$,
\item $\delta_{g,h}:M\to [0,\infty)$ denotes a certain zeroth order deviation of the metrics from each other (cf. section \ref{mainz} below for the definition). 
\end{itemize}
Then the $W_{\pm}(-\Delta_{h},-\Delta_{g}, I_{g,h})$ exist and are complete.}\smallskip

An important feature of this result is that no assumptions on the geometry of $(M,g)$ and $(M,h)$ are imposed, only the deviation of $g$ from $h$ matters, as it should be in scattering theory. The above result has considerably improved an earlier result by Müller and Salomonsen~\cite{salo}, where instead of the zeroth order deviation, the authors had to weight their integral condition with a much stronger second order deviation, in addition to assuming both metrics to have a $C^{\infty}$-bounded geometry.

Nevertheless, a certain drawback of the result from \cite{HPW} is that injectivity radii are very hard to calculate or even to control in general. In any case, one needs a very detailed control on the sectional curvatures, to get some control on the injectivity radius \cite{CGT}. On the other hand, volumes of balls are much more handy and equivalent under quasi-isometry, and in fact any lower bound on the injectivity radius implies a lower bound on the volume function by Bishop-Günther\rq{}s inequality.

In view of these remarks, our main result Theorem \ref{main}, which reads as follows, provides a remarkable improvement: \smallskip

\emph{Assume that $g$ and $h$ are geodesically complete and quasi-isometric Riemannian metrics on $M$ such that for both $j\in\{g,h\}$ and some $s>0$ one has 
 $$
\int \alpha_{j}(x,s) \beta_{j}(x)  \delta_{g,h}(x) d\mu_j(x)<\infty,
$$
where $\alpha_{j}(\cdot,s):M\to [0,\infty)$ is a local upper bound on the heat kernel on $(M,j)$ at the time $s>0$ and $\beta_{j}:M\to [0,\infty)$ is a certain explicitly given local lower bound on $\mathrm{Ric}_j$. Then the wave operators $W_{\pm}(-\Delta_{h},-\Delta_{g}, I_{g,h})$ exist and are complete.}
\smallskip

Again, no assumptions on the geometry of $(M,g)$ are $(M,h)$ are imposed. While Theorem~\ref{main} can be expected to be disjoint from that of \cite{HPW} in general, under global lower Ricci bounds it can be brought into a form which indeed is much more general and handy then the induced result from \cite{HPW} in the sense of the above remarks. In fact, assuming that both Ricci curvatures are bounded from below by constants, one can use Li-Yau type heat kernel estimates and Theorem \ref{main} boils down to give the following criterion (cf.~Corollary~\ref{lower} below):\smallskip

\emph{Assume that $g$ and $h$ are geodesically complete and quasi-isometric Riemannian metrics on $M$ with $\mathrm{Ric}_j$ bounded from below by a constant for both $j\in \{g,h\}$ and
$$
\int  \mu_j(x,1)^{-1}  \delta_{g,h}(x) d\mu_j(x)<\infty \quad\text{ for some $j\in \{g,h\}$,}
$$
where $\mu_j(x,1)$ denotes the volume of the geodesic ball with radius $1$ centered at $x$ with respect to $(M,j)$. Then the wave operators $W_{\pm}(-\Delta_{h},-\Delta_{g}, I_{g,h})$ exist and are complete.}\smallskip

Note that if $g$ and $h$ are geodesically complete and quasi-isometric Riemannian metrics on $M$ with $\mathrm{Ric}_j$ bounded from below by a constant for both $j\in \{g,h\}$, then \eqref{hpw} requires control on some lower bounds of the injectivity radii, while in our Corollary \ref{lower} this condition is replaced by a more general and much more handy lower control on the volume function. The essential difference between our machinery and the one from \cite{HPW} is that we rely on parabolic techniques, while in \cite{HPW} use elliptic estimates. In fact, our main tool is an $L^2\to L^{\infty}_{\mathrm{loc}} $ estimate for the gradient of the heat semigroup that should be of an independent interest, which is valid on every geodesically complete Riemannian manifold, and which relies on an explicit Bismut type probabilistic formula (cf.~\cite{ArnaudonThalmaier, Thalmaier} and the proof of Theorem~\ref{main2} below).

\smallskip

Finally, it is remarkable that the assumptions in Corollary \ref{lower} are explicit enough to deduce the following stability of absolutely continuous spectra under a Ricci flow, which seems to be first result of its kind (cf.~Corollary~\ref{flow}):\smallskip

\emph{Let $S>0$, $\kappa\in\IR$ and assume that
\begin{itemize}
 \item the family $(g_s)_{s\in [0,S]}$ of Riemannian metric on $M$ evolves under a Ricci type flow 
$$
\partial_s g_s=\kappa \mathrm{Ric}_{g_s}, \quad s\in [0,S],
$$
\item the initial metric $g_0$ is geodesically complete 
\item setting
$$
A(x):=\sup\big\{\left|\mathrm{Ric}_{g_s}(v,v)\right|_{g_s}:s\in [0,S], v\in T_x M, |v|_{g_s}\leq 1\big\},\quad x\in M,
$$
one has 
\begin{align}\label{ricci5}
&\sup_{ x\in M}A(x)<\infty,\\\label{ricci6}
&\int \mu_{g_0}(x,1)^{-1}  \sinh\big( (m/4)|\kappa| A(x)S\big) d\mu_{g_0}(x)<\infty.
\end{align}
\end{itemize}
Then one has $\mathrm{spec}_{\mathrm{ac}}(H_{g_s})=\mathrm{spec}_{\mathrm{ac}}(H_{g_0})$ for all $s\in [0,S]$.} 
\smallskip

Note that assumption (\ref{ricci5}) is natural in this context: for example, a typical short time existence result for Ricci flow by \cite{shi} Shi requires that $g_0 $ is geodesically complete with bounded sectional curvatures, yielding a solution $(g_s)_{s\in [0,S^*)}$ of the Ricci flow equation
$$
\partial_s g_s=-2 \mathrm{Ric}_{g_s}, \quad s\in [0,S^*),
$$
which exists up to a time $S^*>0$ and satisfies 
$$
\sup_{s\in [0,S^*),x\in M}\left|\mathrm{Sec}_{g_r}(x)\right| <\infty.
$$
The latter finiteness clearly implies (\ref{ricci5}) for every $S<S^*$.

\section{Main results}\label{mainz}

Let $M$ be a smooth connected manifold of dimension $m\geq 2$. We stress the fact that we understand all our spaces of functions on $M$ (or more generally, all our spaces of sections in vector bundles over $M$) to be complex-valued. For example, $\Omega^1_{\ICC}(M)$ stands for the smooth complex-valued $1$-forms on $M$, that is, the smooth sections of $T^*M\otimes \IC\to M$, and then 
$$
\Id :\ICC(M)\longrightarrow  \Omega^1_{\ICC}(M)
$$
 stands for the complexification of the usual exterior derivative. It will be convenient to set
\begin{align*}
&\mathscr{M}(M):=\big\{\text{smooth Riemannian metrics on $M$}\big\},\\
&\widetilde{\mathscr{M}}(M):=\big\{g\in \mathscr{M}(M): \text{$g$ is geodesically complete}\big\}.
\end{align*}

 Given $g\in \mathscr{M}(M)$ we denote by $\mathrm{Ric}_g$ its Ricci curvature, and by $\mu_g$ the Riemannian volume measure, by $B_g(x,r)$ the open geodesic balls, and by $\mu_g(x,r):=\mu_g(B_g(x,r))$ the volume function. The induced metric on $T^*M$ will be denoted by $g^*$. Complexifications of these data will be denoted by $g_{\IC}$ etc.

The complex Hilbert space $L^2(M,g)$ is given by $\mu_g$-equivalence classes of Borel functions $f:M\to\IC$ with $\int |f|^2 d\mu_g$ finite, and
$$
\left\langle \psi_1,\psi_2\right\rangle_{L^2(M,g)}=\int \overline{\psi_1}\psi_2d\mu_g. 
$$
The complex Sobolev space $W^{1,2}_0(M,g)\subset L^2(M,g)$ is defined to be the closure of $C^{\infty}_c(M)$ with respect to the scalar product
\begin{align}\label{sob}
 \left\langle \psi_1,\psi_2 \right\rangle_{W^{1,2}_0(M,g)}=\big(\int  \overline{\psi_1}\psi_2 \Id\mu_g\big)^{1/2}+ \big(\int g^*_{\IC} (\Id\psi_1,d\psi_2) \Id\mu_g\big)^{1/2},\quad \psi\in C^{\infty}_c(M).
\end{align}
Let $H_{g}\geq 0$ denote the Friedrichs realization of the Laplace-Beltrami operator $- \Delta_{g}\geq 0$ in $\IL^2(M,g)$. We will also need the operator $\Id_g$, which denotes the minimal extension of the exterior differential $\Id$ with respect to $g$. In other words, $\Id_g$ is the closed unbounded operator from $\IL^2(M,g)$ to $\Omega^{1}_{\IL^2}(M,g)$ which is defined by $\dom(\Id_g)=W^{1,2}_0(M,g)$, and $\Id_gf:=\Id f$, in the distributional sense. In fact, one has $H_g=  \Id_g^{*}\Id_g$. If $g$ is geodesically complete, then $H_g$ is essentially self-adjoint on $C^{\infty}_{\c}(M)$.  

Given $g,h\in\mathscr{M}(M)$, we can define a smooth vector bundle morphism 
\begin{align*}
\mathscr{A}_{g,h}:\IT^* M\otimes \IC\longrightarrow  \IT^*  M\otimes \IC,\>\>g^*_{\IC}(\mathscr{A}_{g,h}(x)\alpha ,\beta  ) :=h^*_{\IC}(\alpha ,\beta)\>,\>\>\alpha,\beta \in T^*_xM\otimes \IC, \>x\in M.
\end{align*}
The endomorphism $\mathscr{A}_{g,h}$ is fiberwise self-adjoint with respect to $g^*_{\IC}$, in view of 
$$
g^*_{\IC}(\mathscr{A}_{g,h}(x)\alpha,\alpha )\in\IR\quad\text{ for all $x\in M$, $\alpha\in T^*_xM\otimes \IC$}. 
$$
In addition, $\mathscr{A}_{g,h}$ has fiberwise strictly positive eigenvalues. We further define 
\begin{align*}
\delta_{g,h}:M\longrightarrow [0,\infty),\quad\delta_{g,h}(x):=2\sinh\big((m/4)\max_{\lambda \in\mathrm{spec}(\mathscr{A}_{g,h}(x))}|\log (\lambda) |\big).
\end{align*}
The function $\delta_{g,h}$ measures a $0$-th order deviation of the metrics when we consider them as multiplicative perturbations of each other. We have 
$$
\Id \mu_{h}=\rho_{g,h}\Id \mu_{g} \quad\text{ with a Radon-Nikodym density $0<\rho_{g,h}\in C^{\infty}(M)$}, 
$$
where we record the following simple facts:
\begin{align}\label{apos}
\rho_{h,g}=1/\rho_{g,h},\quad \IAA_{h,g}=\IAA_{g,h}^{-1},\quad \rho_{g,h}=\mathrm{det}(\IAA_{g,h})^{-\f{1}{2}},\quad \delta_{g,h}=\delta_{h,g}.
\end{align}

We write $g\sim h$, if $h$ is quasi-isometric to $g$, that is, if there exists a constant $C\geq 1$ such that 
$$
(1/C)g \leq h \leq Cg \quad\text{ pointwise, as bilinear forms.}
$$
Let us see how these definitions works in the case of conformal perturbations:

\begin{Example} Assume $h= \exp(-(4/m)\phi)g$ for some smooth function $\phi:M\to \IR$, that is, $h$ is a conformal perturbation of $g$. Then one has $h^*= \exp((4/m)\phi)g^*$ and $g\sim h$ holds if and only if $\phi$ is bounded, and then one has $\delta_{g,h}=2\sinh(|\phi|)$. The scattering theory of conformal perturbations has been studied in detail in \cite{BGM}.
\end{Example}

So assume $g\sim h$ for the moment. Then there exists the trivial bounded linear and bijective identification operator
$$
I_{g,h}: \IL^2(M,g)\longrightarrow \IL^2(M,h),\>\>f\longmapsto f,
$$
and one has
\begin{align} \label{apos2}
0<\inf\rho_{g,h}\leq \sup\rho_{g,h}<\infty,\quad\sup\delta_{g,h}<\infty. 
\end{align}
Furthermore, the operator $I^*_{g,h}$ is given by the bounded multiplication operator
\begin{align}\label{adjoint}
I^*_{g,h}: \IL^2(M,g)\longrightarrow \IL^2(M,h),\quad I^*_{g,h}f(x)=\rho_{g,h}(x)f(x).
\end{align}



For every $g\in \widetilde{\mathscr{M}}(M)$, with
$$
(P^g_t)_{t>0}:=(\exp(-tH_g))_{t>0}\subset \ILL(\IL^2(M,g))
$$
the heat semigroup defined by the spectral calculus\footnote{In the sequel, whenever a Borel equivalence class of $L^2$-functions on $(M,g)$ has a smooth representative, we implicitly take the latter.}, and $(x,s)\in M\times (0,\infty)$, we define the finite quantities
\begin{align*}
&\Psi_{1,g}(x):=\max\Big(0,\inf \big\{C\in \IR:\mathrm{Ric}_g(v,v) \geq C |v|_g^2   \text{ \rm for all $y\in B_g(x,1/2)$, $v\in T_yM$} \big\}\Big),\\
&\Psi_{2,g}(x):= \pi^2(m+3) +\pi \sqrt{\Psi_{1,g}(x)(m-1)}+4\Psi_{1,g}(x),\\
&\Psi_{3,g}(x,s):=   \Psi_{2,g}(x)\big(1-\exp(-\Psi_{2,g}(x)s)\big)^{-1},\\
&\Psi_{4,g}(x,s):=\sup_{y\in M}P^g_s(x,y).
\end{align*}

While it is well-known \cite{buch} that $\Psi_{4,g}(x,s)<\infty$ for all $(x,s)\in M\times (0,\infty)$, one can even prove \cite{gueneysubook}
$$
\sup_{x\in K}\Psi_{4,g}(x,s)<\infty\quad\text{ for all $s\in (0,\infty)$, $K\subset M$ compact.}
$$



Here comes our main result:

\begin{Theoreme}\label{main} Assume that $g,h\in\widetilde{\mathscr{M}}(M)$ satisfy $g\sim h$ and that for some $s\in (0,\infty)$ and both $j\in\{g,h\}$ one has 
 $$
\int \Psi_{3,j}(x,s) \Psi_{4,j}(x,s)  \delta_{g,h}(x) d\mu_j(x)<\infty.
$$
Then the wave operators 
$$
W_{\pm}(H_{h},H_{g}, I_{g,h})=\slim_{t\to\pm\infty}\exp(itH_{h})I_{g,h}\exp(-itH_{g})P_{\mathrm{ac}}(H_g)
$$
exist and are complete (cf.~Theorem \ref{belo} for the definition of completeness). 

Moreover, $W_{\pm}\big(H_{h},H_g, I_{g,h}\big)$ are partial isometries with inital space $\mathrm{Ran} \: P_{\mathrm{ac}}(H_g)$ and final space $\mathrm{Ran} \: P_{\mathrm{ac}}(H_h)$, and one has $\mathrm{spec}_{\mathrm{ac}}(H_g)=\mathrm{spec}_{\mathrm{ac}}(H_h)$.
\end{Theoreme}

Note that the assumptions and the conclusions of Theorem \ref{main} are symmetric in $(g,h)$. The ultimate definition of the functions $\Psi_{j,g/h}$ is dictated by the bound from Theorem \ref{main2} below.

In case the Ricci curvatures are bounded from below by constants, Theorem \ref{main} can be brought into the following convenient form:


\begin{Corollarye}\label{lower} Assume that $g,h\in\widetilde{\mathscr{M}}(M)$ satisfy the following assumptions,
\begin{itemize}
\item $g\sim h$,
\item $\mathrm{Ric}_j$ is bounded from below by a constant for both $j\in\{g,h\}$, 
\item there exists $j\in\{g,h\}$ with
 $$
\int  \mu_j(x,1)^{-1}  \delta_{g,h}(x) d\mu_j(x)<\infty .
$$
\end{itemize}
Then the wave operators 
$$
W_{\pm}(H_{h},H_{g}, I_{g,h})=\slim_{t\to\pm\infty}\exp(itH_{h})I_{g,h}\exp(-itH_{g})P_{\mathrm{ac}}(H_g)
$$
exist and are complete. Moreover, $W_{\pm}\big(H_{h},H_g, I_{g,h}\big)$ are partial isometries with inital space $\mathrm{Ran} \: P_{\mathrm{ac}}(H_g)$ and final space $\mathrm{Ran} \: P_{\mathrm{ac}}(H_h)$, and one has $\mathrm{spec}_{\mathrm{ac}}(H_g)=\mathrm{spec}_{\mathrm{ac}}(H_h)$. 
\end{Corollarye}

\begin{proof} Firstly note that if one has
 $$
\int  \mu_j(x,1)^{-1}  \delta_{g,h}(x) d\mu_j(x)<\infty 
$$
for some $j\in \{g,h\}$, then by quasi-isometry the same is true for both $j\in \{g,h\}$.

One trivially has 
$$
\sup_{x\in M}\Psi_{3,j}(x,1)<\infty
$$
and by a Li-Yau type heat kernel estimate \cite{sturm} there exist $a,b\in (0,\infty)$, which only depend on the lower Ricci bounds and $m$, such that for all $x\in M$ one has $$\Psi_{4,j}(x,1)\leq a\mu_j(x,1)^{-1}\exp(b),$$ so that the claim follows from Theorem \ref{main}.
\end{proof}

Corollary \ref{lower} implies the following result concerning the stability of the absolutely continuous spectrum of a family of metrics that evolve under a Ricci flow, as long as the initial metric has a bounded Ricci curvature:

\begin{Corollarye}\label{flow} Let $S>0$, $\kappa\in\IR$ and assume that
\begin{itemize}
 
\item the family $(g_s)_{s\in [0,S]}\subset \IMM(M)$ evolves under a Ricci type flow 
$$
\partial_s g_s=\kappa \mathrm{Ric}_{g_s}, \quad s\in [0,S],
$$
\item the initial metric $g_0$ is geodesically complete 
\item setting
$$
A(x):=\sup\big\{\left|\mathrm{Ric}_{g_s}(v,v)\right|_{g_s}:s\in [0,S], v\in T_x M, |v|_{g_s}\leq 1\big\},\quad x\in M,
$$
one has 
\begin{align}\label{ricci1}
& \sup_{ x\in M}A(x)<\infty,\\\label{ricci2}
&\int \mu_{g_0}(x,1)^{-1}  \sinh\big( (m/4)\,S\,|\kappa| A(x)\big) d\mu_{g_0}(x)<\infty.
\end{align}
\end{itemize}
Then one has $\mathrm{spec}_{\mathrm{ac}}(H_{g_s})=\mathrm{spec}_{\mathrm{ac}}(H_{g_0})$ for all $s\in [0,S]$.
\end{Corollarye}

\begin{proof} Let $s\in [0,S]$. It is well-known that the Ricci flow equation in combination with (\ref{ricci1}) imply $g_s\sim g_0$, so that in particular all $g_s$ are geodesically complete \cite{topping}. We give the simple proof for the convenience of the reader: Set 
$$
A:=\sup_{  x\in M}A(x)
$$
 and assume $x\in M, v\in T_xM$. If $\kappa>0$ we have
\begin{align*}
&\partial_s g_s(v,v)=\kappa \mathrm{Ric}_{g_s}(v,v)\leq \kappa A g_s(v,v),\\
&\partial_s g_s(v,v)=\kappa \mathrm{Ric}_{g_s}(v,v)\geq -\kappa A g_s(v,v),
\end{align*}
so that from Gronwall\rq{}s Lemma we get
\begin{align*}
&g_s(v,v)\leq \exp(sA\kappa )g_0(v,v),\\
&g_s(v,v)\geq \exp(-sA\kappa )g_0(v,v).
\end{align*}
In case $\kappa<0$ we have
\begin{align*}
\partial_s g_s(v,v)=\kappa \mathrm{Ric}_{g_s}(v,v)\leq -\kappa A g_s(v,v),\\
\partial_s g_s(v,v)=\kappa \mathrm{Ric}_{g_s}(v,v)\geq \kappa A g_s(v,v),
\end{align*} 
again Gronwall gives the asserted quasi-isometries.

 It remains to prove that for all $s$ the integrability (\ref{ricci2}) implies 
$$
\int  \mu_{g_0}(x,1)^{-1} \delta_{g_s,g_0}(x) d\mu_{g_0}(x)<\infty .
$$
To this end, assume again $\kappa>0$ first. Given $x\in M$, $\alpha\in T^*_xM$ we have
\begin{align*}
&\partial_s g_0^*(\mathscr{A}_{g_0,g_s}(x)\alpha,\alpha)= \partial_s g^*_s(\alpha,\alpha)=\partial_s g_s(\alpha^{\sharp_s},\alpha^{\sharp_s})\\ 
&\leq A(x) \kappa g_s(\alpha^{\sharp_s},\alpha^{\sharp_s}) 
=A(x) \kappa g^*_s(\alpha ,\alpha )
=A(x) \kappa   g_0^*(\mathscr{A}_{g_0,g_s}(x)\alpha,\alpha)
\end{align*}
and likewise
$$
\partial_s g^*_0(\mathscr{A}_{g_0,g_s}(x)\alpha,\alpha)\geq -  A(x)\kappa g_0^*(\mathscr{A}_{g_0,g_s}(x)\alpha,\alpha),
$$
so that by Gronwall one has
$$
|\mathrm{log}(\lambda)| \leq A(x)\kappa s\leq A(x)\kappa S\quad\text{ for all $\lambda \in \mathrm{spec}(\mathscr{A}_{g_0,g_s}(x))=\mathrm{spec}(\mathscr{A}_{g_0,g_s}(x)|_{T^*_xM}).$}
$$
The same proof gives in case $\kappa<0$ the inequality  
$$
|\mathrm{log}(\lambda)| \leq -A(x)\kappa s\leq  -A(x)\kappa S\quad\text{ for all  $\lambda \in \mathrm{spec}(\mathscr{A}_{g_0,g_s}(x))=\mathrm{spec}(\mathscr{A}_{g_0,g_s}(x)|_{T^*_xM})$},
$$
showing altogether that 
$$
\delta_{g_s,g_0}(x)\leq 2\sinh\big( (m/4)\,S\,|\kappa|\, A(x)\big),
$$
and completing the proof.
\end{proof}

The operators
$$
(\widehat{P}^g_s)_{s>0}:=(d_g P^g_s)_{s>0}\subset \ILL\big(\IL^2(M,g), \Omega^1_{L^2}(M,g)\big)
$$
will play a crucial role in proof of Theorem \ref{main}. In fact, the main ingredient of the proof is the parabolic gradient bound for the jointly smooth integral kernel 
$$
(0,\infty)\times M\times M \ni (s,x,y)\longmapsto \widehat{P}^g_s(x,y)\in T^*_x M
$$
of $\widehat{P}^g_s$ from Theorem \ref{main2} below, which is certainly of an independent interest. Note that $(s,x,y)\mapsto \widehat{P}^g_s(x,y)$ is the uniquely determined smooth map such that for all $(s,x)\in (0,\infty)\times M$, $f\in L^2(M,g)$ one has 
$$
\widehat{P}^g_sf(x)=\int \widehat{P}^g_s(x,y)f(y) d\mu_g(y).
$$

\begin{Theoreme}\label{main2} For every $g\in \widetilde{\mathscr{M}}(M)$, $(s,x)\in (0,\infty)\times M$ one has 
$$
\int\Big|\widehat{P}^g_s(x,y)\Big|^2_{g^*}d\mu_g(y)\leq   \Psi_{3,g}(x,s) \Psi_{4,g}(x,s).
$$
\end{Theoreme}

\begin{Remark}\label{riesz}
Note that by Riesz-Fischer's duality theorem, the estimate from Theorem \ref{main2} is equivalent to the following statement: For every $g\in \widetilde{\mathscr{M}}(M)$, $(s,x)\in (0,\infty)\times M$, $f\in L^2(M,g)$ one has 
$$
\Big|\widehat{P}^g_sf(x)\Big|_{g^*}\leq \sqrt{\Psi_{3,g}(x,s)  \Psi_{4,g}(x,s)} \left\|f\right\|_{L^2(M,g)}.
$$
\end{Remark}

 \section{Proof of Theorem \ref{main2}}

Here we give the  

\begin{proof}[Proof of Theorem \ref{main2}.]
 
We can omit $g$ in the notation. Let $X(x)$ be a Brownian motion on $M$ starting from $x$ and
$\zeta(x)$ its maximal lifetime. Let us first assume $f$ is real-valued. Then, for $v\in T_xM$ one has the Bismut type formula (cf. Theorem 6.2 in \cite{driver}, Formula (6.2) in \cite{ThalmaierWang98}, \cite{Thalmaier}, \cite{ArnaudonThalmaier})
\begin{equation}\label{Eq:GradFormula}
  \left(dP_{s}f\right)_{x}v=- \mathbb{E}\left[ f(X_s(x))1_{\{s<\zeta(x)\}}\int_{0}^{\tau(x)\wedge s}\big( \Theta_{r}(x)
    \dot{\ell}_{r}(v),dW_{r}(x)\big) \right],
\end{equation}
where
\begin{itemize}
\item $\Theta_r(x)$, $r<\zeta(x)$, is the $\mathrm{Aut}(T_xM)$-valued process defined by
    $$
\frac{d}{d r}\Theta_r(x)=-  \Ric_{\partr r(x)}(\nbull\,,\Theta_r)^\sharp   
    ,\quad \Theta_0(x)=\mathrm{id}_{T_xM},
$$
with $\partr r(x):T_xM\to T_{X_r(x)}M$, $r<\zeta(x)$, the stochastic parallel transport along the paths of $X(x)$;

  \item $\tau(x)<\zeta(x)$ is the first exit time of $X(x)$ from $B(x,1/2)$;
  \item $W(x)$ is a Brownian motion in $T_xM$;
  \item $\ell(v)$ is any adapted process in $T_xM$ with absolutely
    continuous paths such that (for some $\varepsilon>0$)
$$
\ell_0(v)=v, 
\quad\ell_{\tau_{x}}(v)=0\quad \text{and}\quad \mathbb{E}\left[\Bigl(\int_0^{\tau(x)\wedge
  s}\vert\dot\ell_t(v)\vert^2\,dt\Bigr){}^{1/2+1/\epsilon}\right]<\infty.
	$$
\end{itemize}

In fact, the smooth representative of $P_sf(x)$ is given by
$$
P_sf(x)=\E[f(X_s(x) 1_{\{s<\zeta(x)\}}]=\int   P_s(x,y) f(y) d\mu(y).
$$

By Cauchy-Schwarz we obtain
\begin{equation}\label{Eq:GradEst}
  \left|\left(dP_{s}f\right)_{x}v\right|\leq \left(\E\left[ |f|^2(X_s(x)) 1_{\{s<\zeta(x)\}}\right]\right)^{1/2}
  \left(\E\left[\left(\int_{0}^{\tau(x)\wedge s}\big(\Theta_{r}(x)
        \dot{\ell}_{r}(v),dW_{r}(x)\big) \right)^2\right]\right)^{1/2}.
\end{equation}
It is well-known how to estimate the second factor on the right by choosing $\ell(v)$ appropriately,
e.g.~\cite[Sect.~4]{ThalmaierWang2011}. Namely, for $|v|\leq 1$ one can achieve
\begin{equation}
 \E\left[\left(\int_{0}^{\tau(x)\wedge s}\big( \Theta_{r}(x)
        \dot{\ell}_{r}(v),dW_{r}(x)\big) \right)^2\right]   \leq  \Psi_3(x,s) .
\end{equation}

Thus
\begin{align*} 
  \left|\left(dP_{s}f\right)_{x}v\right|&\leq \left(\E\left[ |f|^2(X_s(x)) 1_{\{s<\zeta(x)\}}\right]\right)^{1/2}\sqrt{\Psi_3(x,s)}\\
	&= \left(\int |f(y)|^2 P_s(x,y)d\mu (y)\right)^{1/2} \sqrt{\Psi_3(x,s)}\\
&\leq \sqrt{\Psi_4(x,s)}\sqrt{\Psi_3(x,s)}\left\|f\right \|_{L^2(M)}.
\end{align*}
Using that complexifications are norm preserving and using Remark \ref{riesz}, the latter bound completes the proof.
\end{proof}

\section{Proof of Theorem \ref{main} }

We start by noting that given a diagonizable linear operator $A$ on a finite dimensional linear space and a real-valued function $f$ on the spectrum of $A$, the linear operator $f(A)$ can be defined using the projectors onto the eigenspaces of $A$. In particular, this procedure does not depend on a scalar product, but if $A$ is self-adjoint w.r.t. some scalar product, then the above definition of $f(A)$ is consistent with the spectral calculus.

For example, if we are given $g,h\in\mathscr{M}(M)$ and a point $x\in M$, then 
$$
 \varrho_{g,h}(x)\mathscr{A}_{g,h}(x): T^*_xM\otimes\IC \longrightarrow T^*_xM\otimes\IC
$$
is diagonizable. We define a function and a Borel section in $ T^*M\otimes \IC\to M$ by setting
\begin{align*}
&S_{g,h}:M\longrightarrow \IR, \quad S_{g,h}(x):=\varrho_{g,h}(x)^{1/2}-\varrho_{g,h}(x)^{-1/2},\\
&\widehat{S}_{g,h}:M\longrightarrow \mathrm{End}(\IT^* M\otimes \IC),\quad\widehat{S}_{g,h}(x):=\big(\varrho_{g,h}(x)\mathscr{A}_{g,h}(x)\big)^{1/2}-\big(\varrho_{g,h}(x)\mathscr{A}_{g,h}(x)\big)^{-1/2}.
\end{align*}
One has the elementary bounds (cf. Lemma 3.3 in \cite{HPW})
\begin{align}\label{bound}
\max\big(\left|S_{g,h}(x)\right| , \big|\widehat{S_{g,h}}(x)\big|_{g} ,\big|\widehat{S_{g,h}}(x)\big|_{h} \big)\leq \delta_{g,h}(x)\quad\text{for all $x\in M$},
\end{align}
in particular, the assumption $\sup\delta_{g,h}<\infty$ (which is equivalent to $g\sim h$) implies the boundedness of $S_{g,h}$, $\widehat{S}_{g,h}$. We will need the maximally defined multiplication operators 
\begin{align*}
&S_{g,h;j}:\IL^2(M,j)\longrightarrow \IL^2(M,j),\quad S_{g,h;j}f(x)=|S_{g,h}(x)|^{\f{1}{2}}f(x),\quad j=g,h,\\
&\widehat{S_{g,h;j}}:\Omega^1_{\IL^2}(M,j )\longrightarrow \Omega^1_{\IL^2}(M,j ),\quad \widehat{S_{g,h;j}}f(x)=|\widehat{S_{g,h}}(x)|^{\f{1}{2}}f(x),\quad j=g,h,\\
&U_{g,h}:\IL^2(M,g)\longrightarrow \IL^2(M,h),\quad U_{g,h}f(x)=\big[\mathrm{sgn}(S_{g,h})\varrho_{g,h}^{-\f{1}{2}}\big](x)f(x),\\
&\widehat{U_{g,h}}:\Omega^1_{\IL^2}(M,g )\longrightarrow\Omega^1_{\IL^2}(M,h),\quad \widehat{U_{g,h}}\alpha(x)=\big[\mathrm{sgn}(\widehat{S_{g,h}})(\varrho_{g,h}\IAA_{g,h})^{-\f{1}{2}}\big](x)\alpha(x). 
\end{align*}

The operators $U_{g,h}$, $\widehat{U_{g,h}}$ are always unitary and self-adjoint, and the operators $S_{g,h;j}$, $\widehat{S_{g,h;j}}$ ($j=g,h$) are always self-adjoint and additionally bounded in case $g\sim h$.

The following technical result provides the link between Theorem \ref{main2} and the proof of Theorem \ref{main}. It is a variant of a decomposition formula by Hempel-Post-Weder (cf. Lemma 3.4 from in \cite{HPW}):


\begin{Lemma}[HPW formula] Let $g,h\in\mathscr{M}(M)$ be given with $g\sim h$. Then defining the bounded operator $T_{g,h,s}:\IL^2(M,g)\to \IL^2(M,h)$ by
\begin{align*}
T_{g,h,s}:=\big(\widehat{S_{g,h;h}}\widehat{P}^h_s\big)^{*}\widehat{U_{g,h}}\widehat{S_{g,h;g}}\widehat{P}^g_s-\big(S_{g,h;h} P^h_s\big)^{*}U_{g,h}S_{g,h;g} P^g_{s/2}H_g P^g_{s/2},
\end{align*}
the following formula holds for all $s>0$, $f_h\in \dom (H_h)$, $f_g\in \dom (H_g)$,
\begin{align*}
\left\langle f_h ,T_{g,h,s}f_g\right\rangle_{\IL^2(M,h)} =\left\langle H_hf_h, P^h_s I_{g,h} P^g_sf_g\right\rangle_{\IL^2(M,h)} -\left\langle f_h,P^h_s I_{g,h}P^g_s H_gf_g\right\rangle_{\IL^2(M,h)}.
\end{align*}

\end{Lemma}

\begin{proof} Adding and subtracting the term $\left\langle I^{-1}_{g,h} P^h_s f_h, H_g P^g_sf_g\right\rangle_{\IL^2(M,g)}$ we get
\begin{align*}
&\left\langle H_hf_h, P^h_s I_{g,h} P^g_sf_g\right\rangle_{\IL^2(M,h)} -\left\langle f_h, P^h_s I_{g,h} P^g_s H_gf_g\right\rangle_{\IL^2(M,h)}\\
&=\left\langle H_hP^h_sf_h,  I_{g,h} P^g_sf_g\right\rangle_{\IL^2(M,h)} -\left\langle I^{-1}_{g,h}P^h_sf_h,   H_g P^g_sf_g\right\rangle_{\IL^2(M,g)}\\
&\quad- \left\langle  P^h_s f_h, \big(I_{g,h}-(I^{-1}_{g,h})^*\big) H_g P^g_sf_g\right\rangle_{\IL^2(M,g)}\\
&=\left\langle \Id_hP^h_sf_h,  \Id_h I_{g,h} P^g_sf_g\right\rangle_{\Omega_{\IL^2}(M,h)} -\left\langle \Id_g I^{-1}_{g,h}P^h_sf_h,   \Id_g P^g_sf_g\right\rangle_{\IL^2(M,g)}\\
&\quad- \left\langle  P^h_s f_h, \big(I_{g,h}-(I^{-1}_{g,h})^*\big) H_g P^g_sf_g\right\rangle_{\IL^2(M,g)}.
\end{align*}
Using (\ref{apos}) and (\ref{adjoint}), the latter is
\begin{align*}
&=\int h^*_{\IC}\Big((1-\varrho_{g,h}^{-1}\IAA_{g,h}^{-1})\Id_h P^h_sf_h,\Id_g P^g_s f_g \Big)  \Id\mu_h-\int \overline{ P^h_s f_h}(1-\varrho_{g,h}^{-1})H_gP^g_s f_g   \Id\mu_h\\
&=\int h^*_{\IC}\Big(\mathrm{sgn}(\widehat{S_{g,h}})(\varrho_{g,h}\IAA_{g,h})^{-\f{1}{2}}|\widehat{S_{g,h}}|^{\f{1}{2}}|\widehat{S_{g,h}}|^{\f{1}{2}}\Id_h P^h_s f_h \ , \ \Id_g P^g_s f_g \Big) \Id\mu_h\\
&\quad-\int \overline{  P^h_s f_h} \cdot \mathrm{sgn}(S_{g,h}) \cdot \varrho_{g,h}^{-\f{1}{2}} \cdot | S_{g,h}|^{\f{1}{2}} \cdot |S_{g,h}|^{\f{1}{2}} \cdot  H_gP^g_s f_g   \Id\mu_h\\
&=\int h^*_{\IC}\Big(\Id_h P^h_s f_h \ , \ \Big[\mathrm{sgn}(\widehat{S_{g,h}})(\varrho_{g,h}\IAA_{g,h})^{-\f{1}{2}}|\widehat{S_{g,h}}|^{\f{1}{2}}|\widehat{S_{g,h}}|^{\f{1}{2}}\Big]^{\dagger_h}\Id_g P^g_s f_g \Big) \Id\mu_h\\
&\quad-\int \overline{  P^h_s f_h}   \mathrm{sgn}(S_{g,h})   \varrho_{g,h}^{-\f{1}{2}}   | S_{g,h}|^{\f{1}{2}}   |S_{g,h}|^{\f{1}{2}}   H_gP^g_s f_g   \Id\mu_h\\
&=\int h^*_{\IC} \Big(\Id_h P^h_s f_h \ , \ |\widehat{S_{g,h}}|^{\f{1}{2}}\mathrm{sgn}(\widehat{S_{g,h}})(\varrho_{g,h}\IAA_{g,h})^{-\f{1}{2}}|\widehat{S_{g,h}}|^{\f{1}{2}}\Id_g P^g_s f_g \Big)  \Id\mu_h\\
&\quad-\int \overline{  P^h_s f_h}   | S_{g,h}|^{\f{1}{2}} \mathrm{sgn}(S_{g,h}) \varrho_{g,h}^{-\f{1}{2}}     |S_{g,h}|^{\f{1}{2}}   H_gP^g_s f_g   \Id\mu_h\\
&= \left\langle  \Id_h P^h_s f_h \ , \ \widehat{S_{g,h;h}} \widehat{U_{g,h}} \widehat{S_{g,h;g}}\Id_g P^g_s f_g   \right\rangle_{\Omega^1_{\IL^2}(M,h)}\\
&\quad- \left\langle  P^h_s f_h,      S_{g,h;h} U_{g,h}S_{g,h;g}  H_gP^g_s f_g    \right\rangle_{\IL^2(M,h)}\\
&= \left\langle   f_h \ , \ (\Id_h P^h_s)^* \widehat{S_{g,h;h}} \widehat{U_{g,h}} \widehat{S_{g,h;g}}\Id_g P^g_s f_g   \right\rangle_{ \IL^2 (M,h)}\\
&\quad- \left\langle   f_h,    P^h_s  S_{g,h;h} U_{g,h}S_{g,h;g} P^g_{s/2} H_g P^g_{s/2} f_g    \right\rangle_{\IL^2(M,h)},
\end{align*}
proving the claimed formula.
 \end{proof}

We can now give the actual proof of Theorem \ref{main}:

\begin{proof}[Proof of Theorem \ref{main}] We are going to check the assumptions of Belopol\rq{}skii-Birman's Theorem (cf. Theorem \ref{belo}):

Firstly, by $g\sim h$, the operator $I:=I_{g,h}$ is well-defined and bounded, with a bounded inverse $I^{-1}= I_{h, g}$. The sesquilinear form corresponding to $H_j$ has domain $W^{1,2}_0(M,j)$ for $j=g,h$, so that in view of (\ref{sob}), (\ref{apos}), (\ref{apos2}), the assumption $g\sim h$ also implies $IW^{1,2}(M,h)=W^{1,2}(M,h)$. 

Next, we claim that $( I^*I-1)P^g_s$ is Hilbert-Schmidt (and thus compact) for some $s>0$. Indeed, by (\ref{adjoint}) the operator $( I^*I-1)P^g_s$ has the integral kernel
\begin{align*}
\left[(I^*I-1)P^g_s  \right](x,y)&=(\varrho_{g,h}(x)-1)P^g_s  (x,y)\\
&=\varrho_{g,h}^{-1/2}\mathrm{sgn}(S_{g,h})|S_{g,h}|^{1/2}|S_{g,h}|^{1/2}P^g_s   (x,y),
\end{align*}
so that using $\int P^g_s  (x,y)d\mu_g(y)\leq 1$ we get the bounds
\begin{align*}
&\int \Big|\left[(I^*I-1)P^g_s  \right](x,y)\Big|^2\Id \mu_g(y)\\
&\leq \left\|\varrho_{g,h}^{-1}\, S_{g,h}\right\|_{\infty} |S_{g,h}(x)|  \int P^g_s  (x,y)^2d\mu_g(y)\\
&\leq \left\|\varrho_{g,h}^{-1}\, S_{g,h}\right\|_{\infty}  |S_{g,h}(x)| \Psi_{4,g}(x,s)\int P^g_s  (x,y)d\mu_g(y)
\end{align*}
for some $s>0$. Using (\ref{bound}) we arrive at the following Hilbert-Schmidt estimate,
\begin{align*}
&\int\int \Big|\left[(I^*I-1)P^g_s  \right](x,y)\Big|^2\Id \mu_g(y)\Id \mu_g(x)\\
&\lesssim\int \delta_{g,h}(x) \Psi_{4,g}(x,s)d\mu_{g}(x)\\
&\lesssim\int \delta_{g,h}(x) \Psi_{4,g}(x,s) \Psi_{3,g}d\mu_{g}(x)<\infty,
\end{align*}
as $\Psi_{4,j}(x,s)\geq 1$. Next, as the product of Hilbert-Schmidt operators is compact, the HPW formula shows that it remains to show that the operators $\widehat{S_{g,h;j}}\widehat{P}^j_s$ and $S_{g,h;j} P^j_s$ are Hilbert-Schmidt, for $j=g,h$. To see this, as $S_{g,h;j} P^j_s$ has the integral kernel
$$
\left[S_{g,h;j} P^j_s\right](x,y)=S_{g,h}(x) P^j_s(x,y),
$$
it follows as above that
\begin{align*}
\int\int \Big|\left[S_{g,h;j} P^j_s\right](x,y)\Big|^2\Id \mu_g(y)\Id \mu_g(x)&=\int\int \Big| S_{g,h}(x) P^j_s(x,y)\Big|^2\Id \mu_g(y)\Id \mu_g(x)\\
&\lesssim\int \delta_{g,h}(x) \Psi_{4,j}(x,s)d\mu_{j}(x)\\
&\lesssim\int \delta_{g,h}(x) \Psi_{4,j}(x,s)\Psi_{3,j}(x,s)d\mu_{j}(x)<\infty.
\end{align*}

Likewise, in order to prove the Hilbert-Schmidt property of $\widehat{S_{g,h;j}}\widehat{P}^j_s$, we use Theorem \ref{main2}:
\begin{align}
\int \Big|\widehat{P}^j_s(x,y)\Big|^2_{j^*} d\mu_j(y)\leq \Psi_{3,j}(x,s) \Psi_{4,j}(x,s) ,
\end{align}
so that in view of (\ref{bound}) one has
\begin{align*}
\int\int\Big|[\widehat{S_{g,h;j}}\widehat{P}^j_s](x,y)\Big|^2_{j^*}d\mu_j(y)d\mu_j(x)&=\int\int\Big|\widehat{S_{g,h}}(x)\widehat{P}^j_s(x,y)\Big|^2_{j^*}d\mu_j(y)d\mu_j(x)\\
&\lesssim \int\delta_{g,h}(x)\int\Big|\widehat{P}^j_s(x,y)\Big|^2_{j^*}d\mu_j(y)d\mu_j(x)\\
&\leq \int\delta_{g,h}(x) \Psi_{3,j}(x,s) \Psi_{4,j}(x,s)  d\mu_{_j}(x).
\end{align*}
This completes the proof.
\end{proof}

\appendix 
\section{Belopol\rq{}skii-Birman theorem}\label{wave}

We recall \cite{Re3, weidmann2} that given a self-adjoint operator $H$  in a Hilbert space $\IHH$ with its operator valued spectral measure $E_H$, one defines the \emph{$H$-absolutely continuous subspace $\IHH_{\mathrm{ac}}(H)$ of $\IHH$} to be the space of all $f\in \IHH$ such that the Borel measure $\left\|E_H(\cdot)f\right\|^2$ on $\IR$ is absolutely continuous with respect to the Lebesgue measure. Then $\IHH_{\mathrm{ac}}(H)$ becomes a closed subspace of $\IHH$ and the restriction $H_{\mathrm{ac}}$ of $H$ to $\IHH_{\mathrm{ac}}(H)$ is a well-defined self-adjoint operator. The \emph{absolutely continuous spectrum $\mathrm{spec}_{\mathrm{ac}}(H)$ of $H$} is defined to be the spectrum of $H_{\mathrm{ac}}$.\smallskip

We record a version of the abstract Belopol\rq{}skii-Birman theorem for two Hilbert space scattering theory, which is well-suited for our purpose:

\begin{Theorem}\label{belo}\emph{(Belopol\rq{}skii-Birman)} For $k=1,2$, let $H_k\geq 0$ be a self-adjoint operator in a Hilbert space $\IHH_k$, where $E_{k}$ denotes the operator valued spectral measure, and and $P_{\mathrm{ac}}(H_k)$ the projection onto the absolutely continuous subspace of $\IHH_k$ corresponding to $H_k$. Assume that $I\in\ILL(\IHH_1, \IHH_2)$ is a bounded operator such that the following assumptions hold:
\begin{itemize}
\item $I$ has a two-sided bounded inverse
\item One has 
\begin{align*}
(I^*I-1)P^g_sH_1) &\in \IJJ^{\infty}(\IHH_1)\>\>\text{ (compact) for some $s>0$}.
\end{align*}

\item  There exists an operator $T\in\IJJ^1(\IHH_1, \IHH_2)$ (trace class) and a number $s>0$ such that for all $f_2\in\dom(H_2)$, $f_1\in\dom(H_1)$ one has
\begin{align*}
\left\langle f_2 ,Tf_1\right\rangle_{\IHH_2}\>=\>
&\left\langle H_2f_2, \exp(-sH_{2}) I \exp(-sH_{1})f_1\right\rangle_{\IHH_2} \\
&-\left\langle f_2, \exp(-sH_{2}) I \exp(-sH_{1}) H_1f_1\right\rangle_{\IHH_2}. 
\end{align*}

\item One has either $I\dom(\widehat{P}_1)=\dom(\widehat{P}_2)$ or $I\dom(H_1)=\dom(H_2)$.

\end{itemize}
Then the wave operators 
$$
W_{\pm}(H_{2},H_1, I)=\slim_{t\to\pm\infty}\exp(itH_{2})I\exp(-itH_{1})P_{\mathrm{ac}}(H_1)
$$
exist\footnote{$\slim_{t\to\pm\infty}$ stands for the strong limit.} and are complete, where completeness means that 
$$
\left(\mathrm{Ker} \: W_{\pm}(H_{2},H_1, I)\right)^{\perp}=\mathrm{Ran}\:  P_{\mathrm{ac}}(H_1), \quad\overline{\mathrm{Ran} \: W_{\pm}(H_{2},H_1, I)}=\mathrm{Ran}\:  P_{\mathrm{ac}}(H_2).
$$
Moreover, $W_{\pm}\big(H_{2},H_1, I\big)$ are partial isometries with inital space $\mathrm{Ran} \: P_{\mathrm{ac}}(H_1)$ and final space $\mathrm{Ran} \: P_{\mathrm{ac}}(H_2)$, and one has $\mathrm{spec}_{\mathrm{ac}}(H_1)=\mathrm{spec}_{\mathrm{ac}}(H_2)$.
\end{Theorem}

\begin{proof} In view of Theorem XI.13 from \cite{Re3} and its proof, it remains to show that for every bounded interval $\mathbb{I}$ the operator $(I^*I-1)E_1(\mathbb{I})$ is compact, and that there exists a trace class operator $D\in\IJJ^1(\IHH_1, \IHH_2)$ such that for every bounded interval $\mathbb{I}$ and all $f_1,f_2$ as above one has
\begin{align*}
\left\langle \varphi ,Df_1\right\rangle_{\IHH_2}=
\left\langle H_2f_2, E_2(\mathbb{I})IE_1(\mathbb{I})f_1\right\rangle_{\IHH_2} 
-\left\langle f_2, E_2(\mathbb{I})IE_1(\mathbb{I})H_1f_1\right\rangle_{\IHH_2}. 
\end{align*}
However, using that for all self-adjoint operators $A$ and all Borel functions $\phi,\phi\rq{}: \IR\to \IC$ one has
$$
\phi(A)\phi\rq{}(A)\subset (\phi\cdot \phi\rq{}) (A),\quad \dom (\phi(A)\phi\rq{}(A))=\dom (\phi(A)\phi\rq{}(A))\cap \dom (\phi\rq{}(A)),
$$
the required compactness becomes obvious, and furthermore it is easily justified that 
$$
D:=\exp(sH_2)E_2(\mathbb{I})T\exp(sH_1)E_1(\mathbb{I})
$$
has the required trace class property.
\end{proof}

\proof[Acknowledgements]The second author has been supported by the Fonds
National de la Recherche Luxembourg (FNR) under the OPEN scheme
(project GEOMREV O14/7628746).


\begin{thebibliography}{99}



\bibitem{ArnaudonThalmaier} Arnaudon, M. \& Thalmaier, A.: \emph{Li-Yau type gradient estimates and Harnack inequalities by stochastic analysis.} Probabilistic approach to geometry, 29--48, Adv. Stud. Pure Math., 57, Math. Soc. Japan, Tokyo, 2010.


\bibitem{BGM} Bei, F. \& Güneysu, B. \& Müller, J.: \emph{Scattering theory of the Hodge-Laplacian under a conformal perturbation.} 
J. Spectr. Theory 7 (2017), no. 1, 235--267. 

\bibitem{CGT}  Cheeger, J. \& Gromov, M. \& Taylor, M.: \emph{Finite propagation speed, kernel estimates for functions of the Laplace operator, and the geometry of complete Riemannian manifolds.} J. Differential Geom. 17 (1982), no. 1, 15--53.



\bibitem{driver} Driver, B.K. \& Thalmaier, A.: \emph{Heat equation derivative formulas for vector bundles.} J. Funct. Anal. 183 (2001), no. 1, 42--108.






\bibitem{buch} Grigor'yan, A.: \emph{Heat kernel and analysis on manifolds.} AMS/IP Studies in Advanced Mathematics, 47. American Mathematical Society, Providence, RI; International Press, Boston, MA, 2009.












%
%
%

\bibitem{gueneysubook} Güneysu, B.:~\emph{Covariant Schrödinger semigroups on noncompact Riemannian manifolds}. To appear in 2017 as textbook in the Birkhäuser series \emph{Operator theory: advances and applications.}

\bibitem{HPW} Hempel, R. \& Post, O. \& Weder, R.:~\emph{On open scattering channels for manifolds with ends}, J. Funct. Anal. 266 (2014), no. 9, 5526--5583.









 



%
%


%

\bibitem{salo} Müller, W. \& Salomonsen, G.:~\emph{Scattering theory for the Laplacian on manifolds with bounded curvature.} J. Funct. Anal. 253 (2007), no. 1, 158--206.




%




\bibitem{Re3} Reed, M. \& Simon, B.:~\emph{Methods of modern mathematical physics. III. Scattering theory.} Academic Press [Harcourt Brace Jovanovich, Publishers], New York-London, 1979.



\bibitem{saloff} Saloff-Coste, L.:~\emph{Uniformly elliptic operators on Riemannian manifolds.} J. Differential Geom. 36 (1992), no. 2, 417--450. 

\bibitem{shi} Shi, W.-X.: \emph{Deforming the metric on complete Riemannian manifolds.} J. Differential Geom. 30 (1989), no. 1, 223--301.

%


\bibitem{sturm} Sturm, K.-T.: \emph{Heat kernel bounds on manifolds.}
Math. Ann. 292 (1992), no. 1, 149--162. 
%


\bibitem{Thalmaier} Thalmaier, A.: \emph{On the differentiation of heat semigroups and Poisson integrals.} Stochastics Stochastics Rep. 61 (1997), no. 3-4, 297--321.

\bibitem{ThalmaierWang2011} Thalmaier, A. \&  Wang, F.-Y.: \emph{A
    stochastic approach to a priori estimates and {L}iouville theorems
    for harmonic maps}, Bull. Sci. Math. \textbf{135} (2011), no.~6-7,
  816--843.

\bibitem{ThalmaierWang98} Thalmaier, A. \& Wang, F.-Y.:~\emph{Gradient estimates for harmonic functions on regular domains in Riemannian manifolds.} J. Funct. Anal. 155 (1998), no. 1, 109--124.

\bibitem{topping} Topping, P.:~\emph{Lectures on the Ricci flow. London Mathematical Society Lecture Note Series, 325. Cambridge University Press, Cambridge, 2006.}

\bibitem{weidmann2} Weidmann, J.:~\emph{Lineare Operatoren in Hilberträumen. Teil II. Anwendungen.} Mathematische Leit\-fäden. B. G. Teubner, Stuttgart, 2003.


\end{thebibliography}
\end{document}